\documentclass[a4paper,12pt]{amsart}

\usepackage{amssymb}
\usepackage{amsmath}
\usepackage{amsfonts}
\usepackage{graphicx}

\usepackage{amscd,enumerate}
\usepackage{amsfonts}

\input xy
\xyoption{all}

\newtheorem{theorem}{Theorem}[section]

\newtheorem{corollary}[theorem]{Corollary}

\newtheorem{definition}[theorem]{Definition}

\newtheorem{examples}[theorem]{Examples}

\newtheorem{lemma}[theorem]{Lemma}

\newtheorem{proposition}[theorem]{Proposition}
\newtheorem{remark}[theorem]{Remark}

\newcommand{\M}{\mathbb{M}}

\begin{document}


\title{Endomorphism rings of Leavitt path algebras }
\author{G. Aranda Pino, K.M. Rangaswamy, M. Siles Molina}
\maketitle

\begin{abstract}
We investigate conditions under which the endomorphism ring of the Leavitt path algebra $L_{K}(E)$ possesses various ring and module-theoretical properties such as being von Neumann regular, $\pi$-regular, strongly $\pi$-regular  or self-injective. We also describe conditions under which $L_{K}(E)$ is continuous as well as automorphism invariant as a right $L_{K}(E)$-module.
\end{abstract}

\section{Introduction}
The notion of  Leavitt path algebra  $L_{K}(E)$ over a graph $E$ and a field $K$ was introduced and initially studied in \cite{AA1} and  \cite{AMP}, both as a generalization of the Leavitt algebras of type $(1, n)$ and as an algebraic analogue of the graph C$^{\ast}$-algebras.

 Although their history is very recent, a flurry of activity has followed since 2005. The main directions of research include: characterization of algebraic properties of a Leavitt path algebra $L_K(E)$ in terms of graph-theoretic properties of $E$; study of the modules over $L_K(E)$; computation of various substructures (such as the Jacobson radical, the socle and the center); investigation of the relationship and connections with $C^*(E)$ and general $C^*$-algebras; classification programs;  generalization of the constructions and results first from row-finite to countable graphs and, finally, from countable to completely arbitrary graphs; study of their $K$-theory, and others.  For examples of each of these directions see for instance \cite{AAS}.
 
In this paper, our main aim is to study the ring $A$ of endomorphisms of $L_{K}(E)$ as a right $L_{K}(E)$-module.
%considering the Leavitt path algebra $L_{K}(E)$ as a right $L_{K}(E)$-module, we investigate conditions under which the endomorphism ring $A$ of $L_{K}(E)$ possesses various ring-theoretical properties. 
Since $A$ is isomorphic to $L_{K}(E)$ when $E$ is finite, our focus is on the case when $E$ is an infinite graph with infinitely many vertices. 

Because $L_{K}(E)$ has plenty of idempotents (in fact, it is an algebra with local units), and this implies that the same happens to $A$, we investigate and provide characterizing conditions on $A$ of properties in which idempotents play a significant role. This is the case of being von Neumann regular, $\pi$-regular, strongly $\pi$-regular and self-injective. 
As a consequence, we are able to describe conditions under which the Leavitt path algebra $L_{K}(E)$ is continuous as well as automorphism invariant as a right $L_{K}(E)$-module.

The paper is organized as follows. In Section 2 we give the definitions, examples and preliminary results needed. This section is also devoted to the study  of the Jacobson radical of $A$, the non singularity, the relationship among left ideals of $L_K(E)$ and those of $A$, and, finally, to the study of the projectivity and flatness of $A$.

Section 3 is  about von Neumann regularity. Concretely we show in Theorem \ref{Regular Endo} that $A$ is von Neumann regular if and only if $L_K(E)$ is left and right self-injective and von Neumann regular (equivalently $L_K(E)$ is semisimple as a right $L_K(E)$-module) and that in turn is equivalent to $E$ being acyclic and such that every infinite path ends in a sink.  Left weak regularity and $\pi$-regularity of $A$ are also considered.

 Last section considers  strongly $\pi$-regularity and self-injectivity of $A$. In particular, Theorem \ref{st Pi Regular} characterizes strong $\pi$-regularity and strong $m$-regularity of $A$ in terms of properties of the graph $E$ and by a description of the concrete structure of $L_K(E)$ as an algebra. An analogous approach is followed in Theorem \ref{Self-injective} concerning self-injectivity of $A$.

\section{Preliminary Results}

We begin this section by recalling the basic definitions and examples of Leavitt path algebras. Also, we will include some of the graph-theoretic definitions that will be needed later in the paper.
%%%%%%%%%%%

A (directed) {\it graph} $E=(E^{0},E^{1},r,s)$ consists of two sets $E^{0}$ and $E^{1}$ together with maps $r,s:E^{1}\rightarrow E^{0}$. The elements of $E^{0}$ are called \textit{vertices} and the elements of $E^{1}$ \textit{edges}. If $s^{-1}(v)$ is a finite set for every $v\in E^{0}$, then the graph is called \textit{row-finite}. If a vertex $v$ emits no edges, that is, if $s^{-1}(v)$ is empty, then $v$ is called a\textit{\ sink}. A vertex $v$ will be called \emph{regular} if is it neither a sink nor an infinite emitter (i.e. 
$0\neq \vert s^{-1}(v) \vert < \infty $ ).

A \emph{path} $\mu $ in a graph $E$ is a finite sequence of edges $\mu =e_{1}\dots e_{n}$ such that $r(e_{i})=s(e_{i+1})$ for $i=1,\dots ,n-1$. In this case, $n=l(\mu)$ is the \emph{length} of $\mu $; we view the elements of $E^{0}$ as paths of length $0$. For any $n\in {\mathbb N}$ the set of paths of length $n$ is denoted by $E^n$. Also, $\text{Path}(E)$ stands for the set of all paths, i.e., $\text{Path}(E)=\bigcup_{n\in {\mathbb N}} E^n$. We denote by $\mu ^{0}$ the set of the vertices of the path $\mu $, that is, the set $
\{s(e_{1}),r(e_{1}),\dots ,r(e_{n})\}$.

A path $\mu =e_{1}\dots e_{n}$ is \textit{closed} if $r(e_{n})=s(e_{1})$, in which case $\mu $ is said to be {\it based at the vertex} $s(e_{1})$. A closed path $\mu =e_{1}\dots e_{n}$ based at $v$ is a \textit{closed simple path} if $r(e_i)\neq v$ for every $i<n$, i.e., if $\mu$ visits the vertex $v$ once only. The closed path $\mu $ is called a \textit{cycle} if it does not pass through any of its vertices twice, that is, if $s(e_{i})\neq s(e_{j})$ for every $i\neq j$. A graph $E$ is called \emph{acyclic} if it does not have any cycles. 

An \textit{exit} for a path $\mu =e_{1}\dots e_{n}$ is an edge $e$ such that $s(e)=s(e_{i})$ for some $i$ and $e\neq e_{i}$. We say that $E$ satisfies \textit{Condition }(L) if every cycle in $E$ has an exit. A graph satisfies \textit{Condition }(NE) if no cycle has an exit, while we say that it satisfies \textit{Condition }(K) if whenever there is a closed simple path based at a vertex $v$, then there are at least two based at that same vertex.

We define a relation $\geq $ on $E^{0}$ by setting $v\geq w$ if there exists a path $\mu$ in $E$ from $v$ to $w$, that is, $v=s(\mu)$ and $w=r(\mu)$. The tree of a vertex $v$ is the set $T(v)=\{w\in E^0\ |\ v\geq w\}$. We say that there is a \emph{birfurcation} at $v$ if $|s^{-1}(v)|\geq 2$ while we call a vertex $v$ a \emph{line point} if $T(v)$ does not contain bifurcations nor cycles.

We say that an infinite path $e_1 e_2 e_3\dots$ \emph{ends in a sink} or \emph{ends in an infinite sink} if there exists $i$ such that $s(e_i)$ is a line point.

For each edge $e\in E^{1}$, we call $e^{\ast }$ a {\it ghost edge}. We let $r(e^{\ast}) $ denote $s(e)$, and we let $s(e^{\ast })$ denote $r(e)$.

\begin{definition} {\rm Given an arbitrary graph $E$ and a field $K$, the \textit{Leavitt path} $K$\textit{-algebra} $L_{K}(E)$ is defined to be the $K$-algebra generated by a set $\{v:v\in E^{0}\}$ of pairwise orthogonal idempotents together with a set of variables $\{e,e^{\ast }:e\in E^{1}\}$ which satisfy the following
conditions:

(1) $s(e)e=e=er(e)$ for all $e\in E^{1}$.

(2) $r(e)e^{\ast }=e^{\ast }=e^{\ast }s(e)$\ for all $e\in E^{1}$.

(3) (The ``CK-1 relations") For all $e,f\in E^{1}$, $e^{\ast}e=r(e)$ and $
e^{\ast}f=0$ if $e\neq f$.

(4) (The ``CK-2 relations") For every regular vertex $v\in E^{0}$,
\begin{equation*}
v=\sum_{\{e\in E^{1},\ s(e)=v\}}ee^{\ast}.
\end{equation*}}
\end{definition}

An alternative definition for $L_K(E)$ can be given using the extended graph $\widehat{E}$. This graph has the same set of vertices $E^0$ and also has the same edges $E^1$ together with the so-called ghost edges $e^*$ for each $e\in E^1$, which go in the reverse direction to that of $e\in E^1$. Thus, $L_K(E)$ can be defined as the usual path algebra $K\widehat{E}$ subject to the Cuntz-Krieger relations (3) and (4) above.

\smallskip

If $\mu = e_1 \dots e_n$ is a path in $E$, we write $\mu^*$ for the element $e_n^* \dots e_1^*$ of $L_{K}(E)$. With this notation it can be shown that the Leavitt path algebra $L_{K}(E)$ can be viewed as the $K$-vector space spanned by $\{pq^{\ast } \ \vert \ p,q\,  \hbox{are paths in} \,  E\}$. 

If $E$ is a finite graph, then $L_{K}(E)$ is unital with $\sum_{v\in E^0} v=1_{L_{K}(E)}$; otherwise, $L_{K}(E)$ is a ring with a set of local units (i.e., a set of elements $X$ such that for every finite collection $a_1,\dots,a_n\in L_K(E)$, there exists $x\in X$ such that $a_ix=a_i=xa_i$ for every $i$) consisting of sums of distinct vertices of the graph.

Many well-known algebras can be realized as the Leavitt path algebra of a graph. The most basic graph configurations are shown below (the isomorphisms for the first three can be found in \cite{AA1}, the fourth in \cite{S}, and the last one in \cite{AAS2}).

\begin{examples}\label{examples}{\rm The ring of Laurent polynomials $K[x,x^{-1}]$ is the Leavitt path algebra of the graph given by a single loop graph. Matrix algebras ${\mathbb M}_n(K)$ can be realized by the line graph with $n$ vertices and $n-1$ edges. Classical Leavitt algebras $L_K(1,n)$ for $n\geq 2$ can be obtained by the $n$-rose (a graph with a single vertex and $n$ loops). Namely, these three graphs are:

$$
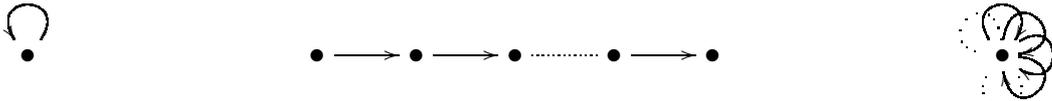
\begin{matrix} \xymatrix{{\bullet} \ar@(ur,ul)} \hskip3cm &
\xymatrix{{\bullet} \ar [r]  & {\bullet} \ar [r]  & {\bullet} \ar@{.}[r] & {\bullet} \ar [r]  & {\bullet} }
\hskip3cm  & \xymatrix{{\bullet} \ar@(ur,dr)  \ar@(u,r)  \ar@(ul,ur)  \ar@{.} @(l,u) \ar@{.} @(dr,dl)
\ar@(r,d) & }
\end{matrix}$$

\medskip

The algebraic counterpart of the Toeplitz algebra $T$ is the Leavitt path algebra of the graph having one loop and one exit: $$
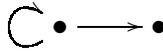
\xymatrix{{\bullet} \ar@(dl,ul) \ar[r] & {\bullet}  }$$

Combinations of the previous examples are possible. For instance, the Leavitt path algebra of the graph

$$
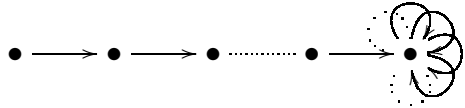
\xymatrix{{\bullet} \ar [r]  & {\bullet} \ar [r]  & {\bullet}
\ar@{.}[r] & {\bullet} \ar [r]  & {\bullet}
 \ar@(ur,dr)  \ar@(u,r)  \ar@(ul,ur)  \ar@{.} @(l,u) \ar@{.} @(dr,dl) \ar@(r,d) & }$$ \smallskip

\noindent is ${\mathbb M}_n(L_K(1,m))$, where $n$ denotes the number of vertices in the graph and $m$ denotes the number of loops.
}
\end{examples}

Another useful property of $L_K(E)$ is that it is a graded algebra, that is, it can be decomposed as a direct sum of homogeneous components $L_K(E)=\bigoplus_{n\in {\mathbb Z}} L_K(E)_n$ satisfying $L_K(E)_nL_K(E)_m\subseteq L_K(E)_{n+m}$. Actually, $$L_K(E)_n=\text{span}_K\{pq^* \vert\ p,q\in \text{Path}(E) , l(p)-l(q)=nÊ \}.$$

Every element $x_n\in L_K(E)_n$ is a homogeneous element of degree $n$. An ideal $I$ is graded if it inherits the grading of $L_K(E)$, that is, if $I=\bigoplus_{n\in {\mathbb Z}} (I\cap L_K(E)_n)$.

%%%%%%%%%%%

\smallskip

We will now outline some easily derivable basic facts about the endomorphism ring $A$ of $L:=L_{K}(E)$. Let $E$ be any graph and $K$ be any field. Denote by $A$ the unital ring $End(L_{L})$.
Then we may identify $L$ with a subring of $A$, concretely, the following is a monomorphism of rings:
$$\begin{matrix}
\varphi: & L & \rightarrow & End(L_{L})\\
            &   x         & \mapsto     &  \lambda_x
\end{matrix}$$ \noindent
where $\lambda_x: L \to L$ is the left multiplication by $x$, i.e., for every $y\in L, $ $\lambda_x(y):= xy$, which is a homomorphism of
right $L$-modules. The map $\varphi$ is also a monomorphism because given a nonzero $x\in L$ there exists an idempotent $u\in L$ such that $xu=x$, hence $0\neq x= \lambda_x(u)$. From now on, and by abuse of notation, we will suppose $L$ to be a subalgebra of $A$.

\medskip

Some of the statements that follow hold not only for Leavitt path algebras but for rings with local units in general. As our main interest in the paper are Leavitt path algebras, we will state them in this context but the reader can rewrite them in general if needed.

\begin{lemma}
\label{Lemma 1)}For any $f\in A$ and any $x\in L$, $f\lambda_{x}=\lambda_{f(x)}\in L$.
\end{lemma}

\begin{proof}
For any $a\in L$, $f\lambda_{x}(a)=f(xa)=f(x)a=\lambda_{f(x)}(a)$.
\end{proof}

\begin{corollary}\label{corollary2.2}
$L$ is a left ideal of $A$.
\end{corollary}

%%%%%%%%%

\begin{lemma}\label{cyclicLPAs}
Let $E$ be a graph and $K$ an arbitrary field. The following are equivalent conditions:
\begin{enumerate}[\rm (i)]
\item\label{Acyclic} The Leavitt path algebra $L_K(E)$ is a cyclic left $A$-module.
\item\label{LKEcyclic} The Leavitt path algebra $L_K(E)$ is a cyclic left $L_K(E)$-module.
\item\label{Efinite} The graph $E$ has a finite number of vertices.
\end{enumerate}
\end{lemma}
\begin{proof} 
(\ref{Acyclic}) $\Leftrightarrow$ (\ref{LKEcyclic}). If $L_K(E)$ is a cyclic left $A$-module then $L_K(E)=Ax$ for some $x\in L_K(E)$. Take a local unit $u$ in $L_K(E)$ such that $ux=x$. Then $L_K(E)=Ax= Aux= L_K(E)x$ (because $L_K(E)$ is a left ideal of $A$, see  Corollary \ref{corollary2.2}), hence $L_K(E)$ is also a cyclic left $L_K(E)$-module. The same reasoning can be used to prove that if $L_K(E)$ is a cyclic left $L_K(E)$-module, then it is a cyclic left $A$-module.

(\ref{LKEcyclic}) $\Leftrightarrow$ (\ref{Efinite}).
Suppose first  $L_K(E)=L_K(E)x$ for some $x\in L_K(E)$. Then $E^0=\{u_1, \dots, u_n\}$, where these vertices $u_i$ are such that $xu_i\neq 0$ for every $i\in \{1, \dots, n\}$.

The converse is easy because if $E^0=\{u_1, \dots, u_n\}$ then the Leavitt path algebra is unital and $u=\sum_{i=1}^nu_i$ is the unit element, so $L_K(E)=L_K(E)u$.
\end{proof}

%%%%%%%%%

Recall (see \cite{GS02}) that given two rings $R\subseteq Q$, we say $Q$ is an \textit{algebra of right quotients }of $R$ if given elements $p,q$ in $Q$ with
$p\neq0$ there exists $r\in R$ such that $pr\neq0$ and $qr\in R$.

\begin{lemma}
\label{Lemma 2}$A$ is an algebra of right quotients of $L$. Further, $I\cap L\neq0$ for any non-zero two-sided ideal $I$ of $A$.
\end{lemma}

\begin{proof}
Consider $f,g\in A$ with $f\neq0$. In particular, $f(u)\neq0$ for some idempotent $u$ in $L$. Then $f\lambda_{u}\neq0$ since $f\lambda_{u}%
(u)=f(u\cdot u)=f(u)\neq0$ and, by Lemma \ref{Lemma 1)}, $f\lambda_{u}=\lambda_{f(u)}\in L$ and $g\lambda_{u}=\lambda_{g(u)}\in L$. Let $0\neq f\in
I$. By the preceding argument there is an idempotent $u\in L$ such that $f\lambda_{u}\in L$. Since $I$ is an ideal, $f\lambda_{u}\in I$. So $I\cap
L\neq0$.
\end{proof}

As a consequence, we get the following

\begin{proposition}
\label{J(A)=0}The Jacobson radical of $A$ is zero.
\end{proposition}

\begin{proof}
Suppose, by way of contradiction, $J(A)\neq0$. From Lemma \ref{Lemma 2}, $J(A)\cap L\neq0$. Let $0\neq r\in J(A)\cap L$. Take an arbitrary $s\in L$.
Since $J(A)$ is an ideal, $sr\in J(A)$; this implies that there exists $q\in A$ such that $sr+q-qsr=0$. Then $q=qsr-sr\in L$ as $L$ is a left ideal of $A$.
Consequently, $r$ belongs to the Jacobson radical $J(L)$. But this leads to a contradiction since $J(L)=0$ by \cite[Proposition 2.2]{AAS}.
\end{proof}

\begin{lemma}
\label{Left Ideal}Every left ideal of $L$ is also a left ideal of $A$.
\end{lemma}

\begin{proof}
Let $I$ be a left ideal of $L$. Consider $f\in A$ and $y\in I$. Since $L$ has local units, there is an idempotent $v\in L$ such that $vy=y$. Then, by Lemma \ref{Lemma 1)}, $f\lambda_{y}=f\lambda_{vy}=\lambda_{f(vy)}=\lambda_{f(v)y}=\lambda_{f(v)}\lambda_{y}\in I$.
\end{proof}

\begin{proposition}
\label{Singular}$A$ is non-singular.
\end{proposition}

\begin{proof}
Suppose the left singular ideal $Z$ of $A$ is non-zero. By Lemma \ref{Lemma 2}, $Z\cap L\neq0$. Let $0\neq y\in Z\cap L$. Now $lan_{A}(y)$ is
an essential left ideal of $A$. Now by Lemma \ref{Lemma 2} and Lemma \ref{Left Ideal}, $lan_{A}(y)\cap L$ is also an essential left ideal of $L$
and since $lan_{A}(y)\cap L=lan_{L}(y)$, $y$ is in the left singular ideal of $L$ which is zero by \cite[Proposition 2.3.8]{AAS}, a contradiction.

Similarly, if the right singular ideal $J$ of $A$ is non-zero, use the fact (Lemma \ref{Lemma 2}) that $A$ is an algebra of right quotients of $L$ to find
a non-zero $x\in J\cap L$ and proceed as in the preceding paragraph to reach a contradiction.
\end{proof}

%%%%%%%

Recall that a ring $R$ is said to be a (left) \textit{exchange ring} if for any direct decomposition $A=M {\textstyle\oplus} N={\textstyle\bigoplus\limits_{i\in I}} A_{i}$ of any left $R$-module $A$, where $R\cong M$ as left $R$-modules and $I$ is a finite set, there always exist submodules $B_{i}$ of $A_{i}$ such that $A=M{\textstyle\oplus} ({\textstyle\bigoplus\limits_{i\in I}} B_{i})$ (see for instance Warfield \cite{W}). In that paper Warfield notes that the property of being an exchange ring is left/right symmetric, so we simple use the term exchange ring.

\par

P. Ara  showed in \cite{Ara1} that a not necessarily unital ring $R$ is an exchange ring if and only if for every element $x\in R$ there exist elements $r,s\in R$ and an idempotent $e\in R$ such that $e=rx=s+x-sx$. 

\par

\begin{proposition}\label{exchange} The following conditions are equivalent:
\begin{enumerate}[{\rm (i)}]
\item $A$ is an exchange ring.
\item $L$ is an exchange ring.
\item $E$ satisfies Condition (K).
\end{enumerate}
\end{proposition}

\begin{proof} (i) $\Longrightarrow$ (ii). Suppose that $A$ is an exchange ring. To show that $L$ is an exchange ring, let $x\in L$. By hypothesis, there are elements $f,g,\epsilon$ in $A$ with $\epsilon$ an idempotent such that $\epsilon=f\lambda_{x}=g+\lambda_{x}-g\lambda_{x}$. Since $L$ is a left ideal of $A$,  $\epsilon=f\lambda_{x}\in L$ and $g=\epsilon\lambda_{x}+g\lambda_{x}\in L$. Moreover, if $u$ is a local unit in $L$ satisfying $ux=x=xu$, then $f\lambda_{x}=f\lambda_{ux}=f\lambda_{u}\lambda_{x}$ and so we can replace $f$ by $f\lambda_{u}=\lambda_{f(u)}\in L$. Hence $L$ is an exchange ring thus proving (ii).
\par

(ii) $\Longrightarrow$ (i). Warfield proves in \cite[Theorem 2]{W} that an $R$-module has the exchange property if and only if its endomorphism ring is an exchange ring. In his paper he considers rings with local units, but the proofs up to the main result \cite[Theorem 2]{W} only use, on the one hand, abstract homological properties of submodules that hold for rings with local units and, on the other hand, a reference to a ``deeper result" by Crawley and J\'onsson \cite[Theorem 7.1]{CJ}. In this last paper, the authors actually consider a much general framework (i.e., algebras in the sense of J\'onsson-Tarski) in which, in the particular case that a ring structure might be considered, the existence of a unit is not assumed at all.

\par

(ii) $\Longleftrightarrow$ (iii). This equivalence has been established in \cite{ARS} for arbitrary graphs.
\end{proof}

%%%%%%%

We have some further consequences of the previous results.

\begin{corollary}\label{projective} $L$ is projective as a left $A$-module.
\end{corollary}
\begin{proof} First write $L={\textstyle\bigoplus\limits_{v\in E^{0}}} Lv$. Now $Lv=Av$ is a direct summand of $A$ as $v$ is an idempotent and hence $Lv$ is projective. Consequently, $L$ is a projective left $A$-module.
\end{proof}

\begin{proposition}\label{purity} $L$ is a pure left submodule of $A$ and, consequently, $A/L$ is a flat $A$-module.
\end{proposition}
\begin{proof} Suppose the system of $m$ equations in $n$ variables $x_{1},\dots,x_{n}$ 
\[{\textstyle\sum\limits_{j=1}^{n}} f_{ij}x_{j}=a_{i}\qquad\qquad(i=1,\dots,m)
\]
where $f_{ij}\in A,a_{i}\in L$, has a solution $x_{j}=g_{j}\in A$ for all $j=1,\dots,n$.  Let $u$ be a local unit in $L$ satisfying $a_{i}u=a_{i}$ for all $i=1,\dots,m$. Then we have, for each $i=1,\dots,m$,   ${\sum_{j=1}^{n}} f_{ij}g_{j}u=a_{i}u=a_{i}$. By Lemma \ref{Lemma 1)}, $g_{j}u=g_{j}\lambda_{u}=\lambda_{g_{j}(u)}\in L$ and so $x_{j}=g_{j}u$, for $j=1,\dots,n$, is a solution of the above system in $L$. This shows that $L$ is pure in $A$. Since $A$ is projective as an $A$-module, $A/L$ is then a flat $A$-module (see \cite{F}).
\end{proof}

%%%%%%%
\section{Von Neumann Regular Endomorphism Rings}

Given a row-finite graph $E$, we give necessary and sufficient conditions under which the endomorphism ring $A$ of $L=L_{K}(E)$ is von Neumann regular.
As an easy application, we describe when $L$ is automorphism invariant as well as when $L$ is continuous as a right $L$-module.

Recall that a ring $R$ is \emph{von Neumann regular} if for every $a\in R$ there exists $b\in R$ such that $a=aba$. The ring $R$ is called $\pi$-regular if for every $a\in R$ there exist $n\in {\mathbb N}$ and $b\in R$ with $a^n=a^nba^n$. A ring $R$ is \emph{left (resp. right) weakly regular} if every left (resp. right) ideal $I$ is idempotent, that is: $I=I^2$. The ring $R$ is said to be \emph{weakly regular} if it is both left and right weakly regular.

\begin{proposition}\label{regularityCond}
Let $E$ be any graph and $K$ be any field. Then:
\begin{enumerate}[\rm (i)]
\item If $A$ is $\pi$-regular then $L_K(E)$ is $\pi$-regular. 
\item \label{A reg => L reg} If $A$ is von Neumann regular then $L_K(E)$ is  von Neumann regular.
\item If $A$ is left weakly regular then $L_K(E)$ is left weakly regular.
\end{enumerate}
\end{proposition}
\begin{proof}
(i). Take $a\in L_K(E)$. Since $A$ is $\pi$-regular there exist $f\in A$ and $n\in \mathbb{N}$ such that 
$(\lambda_a)^n= (\lambda_a)^nf(\lambda_a)^n$. This means $\lambda_{a^n}= \lambda_{a^n}f\lambda_{a^n}=
\lambda_{a^n}\lambda_{f(a^n)}=\lambda_{a^nf(a^n)}$, hence $a^n= a^nf(a^n)$. Choose $u\in L_K(E)$ such that $a=ua$. Then $a^n=a^nf(ua^n)=a^nf(u)a^n$, which shows our claim.

(ii). Let $a\in L$. By hypothesis, there is an $f\in A$ such that $\lambda_{a}=\lambda_{a}f\lambda_{a}$. Choose an idempotent $u\in L$ satisfying
$ua=a=au$ so that $\lambda_{a}=\lambda_{ua}$. Then from Lemma \ref{Lemma 1)} we get $\lambda_{a}=\lambda_{a}f\lambda_{ua}=\lambda_{a}\lambda_{f(ua)}
=\lambda_{a}\lambda_{f(u)a}=\lambda_{a}\lambda_{f(u)}\lambda_{a}$.

(iii). Let $I$ be a left ideal of $R$. By Lemma \ref{Left Ideal} we have that $I$ is also a left ideal of $A$ and so it is idempotent.
\end{proof}

We will be using the following result which was first proved in \cite[Theorem 4]{R} for $\mathbb{Z}$-modules. The same proof holds for arbitrary modules. This was established in \cite{Ware}.

\begin{lemma}\label{Regular=>Im,Ker summand} Let $M$ be a right module over a ring $R$. Then the endomorphism ring of $M$ is von Neumann
regular if and only if both the image and the kernel of every endomorphism of $M$ are direct summands of $M$.
\end{lemma}

%%%%

The proof of the main theorem in this section (Theorem \ref{Regular Endo}) requires a modification of the statements and the proofs of Lemmas 1.2 and 1.3 of H. Bass in \cite{Bass} for Leavitt path algebras. Bass assumes that the ring $R$ he considers has a multiplicative identity, the modules he deals with are left $R$-modules and states his result for free modules. Now, a Leavitt path algebra $L$ does not in general have a multiplicative identity and the modules we consider are right $L$-modules. Moreover $L$ as a right $L$-module need not be a free $L$-module. So we need to modify the arguments of Bass appropriately. Interestingly, as is clear from the following, all the arguments of Bass, after modifications, hold for Leavitt path algebras.

We start with the following assumptions and notation:

Let $E$ be an arbitrary graph, let $L=L_{K}(E)$ and let $p$ be an infinite path in $E$ of the form $p=\gamma_{1}\gamma_{2}\cdots\gamma_{n}\cdots $ where, for all $i\geq1$, $\gamma_{i}$ is a path with $s(\gamma_{i})=v_{i}$ and $r(\gamma_{i})=s(\gamma_{i+1})$. Let $F=({\textstyle\bigoplus\limits_{i=1}^{\infty}}v_{i}L)$ and $L=F\oplus Y$. Let $G={\textstyle\sum\limits_{i=1}^{\infty}}(v_{i}-v_{i+1}\gamma_{i}^{\ast})L$ be the submodule of $F$ generated by the set $\{v_{i}-v_{i+1}\gamma_{i}^{\ast}:i=1,2,...\}$. Suppose $G$ is a direct summand of $F$, say $F=G\oplus H$. For each $k\geq1$, let $v_{k}=g_{k}+h_{k}$, where $g_{k}\in G$ and $h_{k}\in H$.
\medskip

The following result is an adaptation to our context of Bass' Lemma  \cite[Lemma 1.2]{Bass}.

\begin{lemma}
\label{Bass-1}For each integer $k\geq1$, let $J_{k}=\{r\in L \ \vert\ \gamma_{k+n}^{\ast}\cdot\cdot\cdot\gamma_{k}^{\ast}r=0$ for some $n\geq0\}$. Then $J_{k}=(0:h_{k}):=\{r\in L \ \vert\ h_{k}r=0\}$.
\end{lemma}

\begin{proof}
Note that $h_{k}-h_{k+1}\gamma_{k}^{\ast}=(v_{k}-g_{k})-(v_{k+1}-g_{k+1})\gamma_{k}^{\ast}=(v_{k}-v_{k+1}\gamma_{k}^{\ast})-(g_{k}-g_{k+1}\gamma_{k}^{\ast})\in G\cap H=0$. Thus
\[
h_{k}=h_{k+1}\gamma_{k}^{\ast}=h_{k+2}\gamma_{k+1}^{\ast}\gamma_{k}^{\ast}=\cdot\cdot\cdot=h_{k+n+1}\gamma_{k+n}^{\ast}\cdot\cdot\cdot\gamma_{k}^{\ast}=\cdot\cdot\cdot\text{. \qquad\qquad\qquad(I)}
\]
So $\gamma_{k+n}^{\ast}\cdot\cdot\cdot\gamma_{k}^{\ast}r=0$ implies $h_{k}r=0$. Hence $J_{k}\subseteq(0:h_{k})$. Conversely, suppose $r\in(0:h_{k})$ so that $h_{k}r=0$. Then $v_{k}r=g_{k}r\in G$. Therefore we can write $v_{k}r={\textstyle\sum\limits_{i}}(v_{i}-v_{i+1}\gamma_{i}^{\ast})v_{i}r_{i}$, where we may assume, after replacing $r_{i}$ by $v_{i}r_{i}$, that $v_{i}r_{i}=r_{i} $ for all $i$. Moreover, since the sum involves finitely many terms, we may assume that $r_{m}=0$ for sufficiently large $m$. Then comparing the terms on both sides of the equation and using the fact that the submodules $\{v_{i}L \ \vert\ i=1,2,...\}$ are independent, we get $v_{i}r_{i}=r_{i}=0$ for all $i<k$ and $r=r_{k}$. Moreover, using the fact that $v_{j}r_{j}=r_{j}$ for all $j$, we get the following equations:

$0=r_{k+1}-\gamma_{k}^{\ast}r_{k}$

$\vdots$

$0=r_{k+n}-\gamma_{k+n-1}^{\ast}r_{k+n-1}$.

Back-solving for $r_{k+n}$ successively, we get $r_{k+n}=\gamma_{k+n-1}^{\ast}\cdot\cdot\cdot\gamma_{k}^{\ast}r$. Since $r_{k+n}=0$ for sufficiently large $k+n$, for this $k+n$ we will then have $\gamma_{k+n-1}^{\ast}\cdot\cdot \cdot\gamma_{k}^{\ast}r=0$. Hence $r\in J_{k}$.
\end{proof}

\bigskip

The lemma that follows is also inspired by that of Bass \cite[Lemma 1.3]{Bass}.

\begin{lemma}
\label{Bass-2} If $F=G\oplus H$, then the descending chain $L\gamma_{1}^{\ast}\supsetneqq L\gamma_{2}^{\ast}\gamma_{1}^{\ast}\supsetneqq\cdot\cdot\cdot\supsetneqq L\gamma_{n}^{\ast}\cdot\cdot\cdot\gamma_{1}^{\ast}\supsetneqq\cdot\cdot\cdot$ terminates.
\end{lemma}

\begin{proof}
Let $v_{n}=g_{n}+h_{n}$ for all $n$ where $g_{n}\in G$, $h_{n}\in H$. \ Since $F=\bigoplus_{i}v_{i}L$, we can write $h_{n}={\textstyle\sum\limits_{i}}v_{i}c_{in}$, where we can assume, without loss of generality, that $c_{in}=v_{i}c_{in}\in L$ for all $i$ and for all $n$ and that, for each $n$, only finitely many $c_{in}$ are non-zero.

Let $I$ be the left ideal of $L$ generated by the set of coefficients $\{c_{11},c_{21},...,c_{j1},...\}$. From equation (I) we get, for each $j\geq1$, $h_{1}=h_{j+1}\gamma_{j}^{\ast}\cdot\cdot\cdot\gamma_{1}^{\ast}$ which expands to
\[
{\textstyle\sum\limits_{i}}v_{i}c_{i1}={\textstyle\sum\limits_{i}}v_{i}c_{ij+1}\gamma_{j}^{\ast}\cdot\cdot\cdot\gamma_{1}^{\ast}\text{. }
\]
From the independence of the submodules $v_{i}L$ and the fact that $v_{i}c_{ir}=c_{ir}$ for all $i$ and $r$, we get $c_{i1}=c_{ij+1}\gamma_{j}^{\ast}\cdot\cdot\cdot\gamma_{1}^{\ast}\in L\gamma_{j}^{\ast}\cdot\cdot\cdot\gamma_{1}^{\ast}$ for all $i$ so $I\subseteq L\gamma_{j}^{\ast}\cdot\cdot\cdot\gamma_{1}^{\ast}$. This holds for all $j\geq1$. Hence $I\subseteq{\textstyle\bigcap\limits_{j}} L\gamma_{j}^{\ast}\cdot\cdot\cdot\gamma_{1}^{\ast}$. Our goal is to show that for a large $m$, $\gamma_{m}^{\ast}\cdot\cdot\cdot\gamma_{1}^{\ast}\in I$, which will then imply that $I=L\gamma_{m}^{\ast}\cdot\cdot\cdot\gamma_{1}^{\ast}=L\gamma_{m+1}^{\ast}\cdot\cdot\cdot\gamma_{1}^{\ast}=\cdot \cdot\cdot$.

Let $C$ be the column-finite matrix formed by listing, for each $n$, the coefficients $c_{in}$ as entries in column $n$. If $\pi:F\longrightarrow H$ is the coordinate projection mapping $v_{n}$ to $h_{n}$ for all $n$, then the action of $\pi$ is realized by the right multiplication by $C$. Hence $C$ is an idempotent matrix. \ So we get, for all $j$, $c_{j1}={\textstyle\sum\limits_{k=1}^{n+1}}c_{jk}c_{k1}$. Now for all $j$ and for all $k\leq n$, we have $c_{jk}=c_{jn+1}\gamma_{n}^{\ast}\cdot\cdot\cdot\gamma_{1}^{\ast}$ and so $c_{j1}={\textstyle\sum\limits_{k=1}^{n+1}} c_{jn+1}\gamma_{n}^{\ast}\cdot\cdot\cdot\gamma_{1}^{\ast}c_{k1}=c_{jn+1}{\textstyle\sum\limits_{k=1}^{n+1}}\gamma_{n}^{\ast}\cdot\cdot\cdot\gamma_{1}^{\ast}c_{k1}=c_{jn+1}b$ where $b={\textstyle\sum\limits_{k=1}^{n+1}}\gamma_{n}^{\ast}\cdot\cdot\cdot\gamma_{1}^{\ast}c_{k1}\in I$. Now $h_{n+1}b={\textstyle\sum\limits_{j}}v_{j}c_{jn+1}b={\textstyle\sum\limits_{j}}v_{j}c_{j1}=h_{1}$. Then $h_{n+1}(b-\gamma_{n}^{\ast}\cdot\cdot\cdot\gamma_{1}^{\ast})=h_{1}-h_{1}=0$. Hence $b-\gamma_{n}^{\ast}\cdot\cdot\cdot\gamma_{1}^{\ast}\in J_{n+1}$, by Lemma \ref{Bass-1}. This means that for some $m=n+1+t$, $\gamma_{m}^{\ast}\cdot\cdot\cdot\gamma_{n+1}^{\ast}(b-\gamma_{n}^{\ast}\cdot\cdot\cdot\gamma_{1}^{\ast})=0$. Consequently, $\gamma_{m}^{\ast}\cdot\cdot\cdot\gamma_{1}^{\ast}=\gamma_{m}^{\ast}\cdot\cdot\cdot\gamma_{n+1}^{\ast}b\in I$, as $b\in I$. This shows that $I=L\gamma_{m}^{\ast}\cdot\cdot\cdot\gamma_{1}^{\ast}=L\gamma_{m+1}^{\ast}\cdot\cdot\cdot\gamma_{1}^{\ast}=\cdot\cdot\cdot$.
\end{proof}

%%%%

\begin{theorem}\label{Regular Endo} Let $E$ be a row-finite graph and $A$ be the endomorphism ring of $L=L_{K}(E)$ considered as a right $L$-module. Then the following are equivalent:

\begin{enumerate}[\rm (i)]
\item $A$ is von Neumann regular.
\item $E$ is acyclic and every infinite path ends in a sink.
\item $L$ is both left and right self-injective and von Neumann regular.
\item $L$ is a semisimple right $L$-module.
\end{enumerate}
\end{theorem}

\begin{proof}
 (i)$\Longrightarrow$ (ii). Suppose $A$ is von Neumann regular. Then Proposition \ref{regularityCond} (ii) implies that $L$ is von Neumann regular and so $E$ is acyclic by \cite{AR}. We wish to show that every infinite path in $E$ ends in a sink.

Suppose, by way of contradiction, that $p$ is an infinite path with $v_{1}=s(p)$ which does not end in a sink. Since $E$ is acyclic, $p$ will have infinitely
many bifurcating vertices and so we can write $p=\gamma_{1}\gamma_{2}\cdot\cdot\cdot\gamma_{n}\cdot\cdot\cdot$ where, for all $i\geq2$,
$v_{i}=s(\gamma_{i})$ and $e_{i}$ is a bifurcating edge with $s(e_{i})=v_{i}$ so that $e_{i}^{\ast}\gamma_{i}=0$. Consider the descending chain of right ideals
\[
\gamma_{1}L\supseteq\gamma_{1}\gamma_{2}L\supseteq\cdot\cdot\cdot
\supseteq\gamma_{1}\gamma_{2}\cdot\cdot\cdot\gamma_{n}L\supseteq\cdot
\cdot\cdot\text{ .\qquad\qquad\qquad\qquad(II)}
\]
Here, for each $n$, $\gamma_{1}\gamma_{2}\cdot\cdot\cdot\gamma_{n}L\neq \gamma_{1}\gamma_{2}\cdot\cdot\cdot\gamma_{n}\gamma_{n+1}L$. Otherwise, $\gamma_{1}\gamma_{2}\cdot\cdot\cdot\gamma_{n}=\gamma_{1}\gamma_{2}\cdot\cdot\cdot\gamma_{n}\gamma_{n+1}x$ for some $x\in L$. From this we
get $$0\neq e_{n+1}^{\ast}=e_{n+1}^{\ast}\gamma_{n}^{\ast}\cdot\cdot\cdot\gamma_{1}^{\ast}\gamma_{1}\gamma_{2}\cdot\cdot\cdot\gamma_{n}=e_{n+1}^{\ast
}\gamma_{n}^{\ast}\cdot\cdot\cdot\gamma_{1}^{\ast}\gamma_{1}\gamma_{2}\cdot\cdot\cdot\gamma_{n}\gamma_{n+1}x=e_{n+1}^{\ast}\gamma_{n+1}x=0,$$ a
contradiction. Thus (II) is an infinite descending chain of right ideals. Applying the involution map $a\longmapsto a^{\ast}$ on $L$, we get the
following infinite descending chain of left ideals
\[
L\gamma_{1}^{\ast}\supsetneqq L\gamma_{2}^{\ast}\gamma_{1}^{\ast}\supsetneqq\cdot\cdot\cdot\supsetneqq L\gamma_{n}^{\ast}\cdot\cdot\cdot
\gamma_{1}^{\ast}\supsetneqq\cdot\cdot\cdot\text{ .\qquad\qquad\qquad
\qquad(III)}
\]
Write $L=({\textstyle\bigoplus\limits_{i=1}^{\infty}} v_{i}L)\oplus Y$, where $$Y= {\textstyle\bigoplus\limits_{u\in E^{0}\backslash\{v_{i}:i=1,2,...\}}} uL.$$ Let $G={\textstyle\sum\limits_{i=1}^{\infty}} (v_{i}-\gamma_{i}^{\ast})L$ be the submodule generated by the set $\{v_{i}-\gamma_{i}^{\ast}:i=1,2,...\}$. Define an endomorphism $\theta$ of the right $L$-module $L$ by setting $\theta(Y)=0$ and, for all $i=1,2,\cdot\cdot\cdot$, $\theta(v_{i}a)=(v_{i}-\gamma_{i}^{\ast})v_{i}a$, the left multiplication of $v_{i}a$ by $v_{i}-\gamma_{i}^{\ast}$. Since $A$ is von Neumann regular, $G={\rm Im}(\theta)$ is a direct summand of $L$ by Lemma \ref{Regular=>Im,Ker summand}. Lemma \ref{Bass-2}  implies that the descending chain (III) must be finite, a contradiction. Hence every infinite path in $E$ must end in a sink. This proves (ii).

(ii) $\Leftrightarrow$ (iii) $\Leftrightarrow$ (iv) by \cite[Theorem 4.7]{ARS}.

(iv) $\Rightarrow$ (i). If $L$ is a semisimple right module then every submodule and, in particular, the image and kernel of every endomorphism of $L$, is a direct summand of $L$. Then by Proposition \ref{regularityCond} (ii), $A$ is von Neumann regular.
\end{proof}

\begin{corollary}\label{left-right} Let $E$ be a row-finite graph and $K$ be any field. Then $A={\rm End}(L_K(E)_{L_K(E)})$ is von Neumann regular if and only if ${\rm End}({_{L_K(E)}}L_K(E))$ is von Neumann regular.
 \end{corollary}

\begin{remark}{\rm
Our proof of $(i)\Longrightarrow(ii)$ is inspired by the ideas of Bass \cite{Bass} and Ware \cite{Ware}.}
\end{remark}

Leavitt path algebras which are quasi-injective (equivalently self-injective) are described in \cite{ARS}. Recall that a right module $M$ over a ring $R$ is \textit{quasi-injective} if every homomorphism from any submodule $S$ to $M$ extends to an endomorphism of $M$. It is known \cite[Corollary 19.3]{F} that
a module $M$ is quasi-injective if and only if $M$ is invariant under every endomorphism of its injective hull. A generalization of this condition leads to
the concept of an automorphism invariant module (see \cite{ESS}, \cite{GA}). A right module $M$ over a ring $R$ is said to be \textit{automorphism invariant}
if $M$ is invariant under every automorphism of its injective hull $E(M)$. Clearly, a quasi-injective module is automorphism invariant, but the converse
does not hold. As noted in \cite[Example 9]{ESS}, let $R$ be the ring of all eventually constant sequences $(x_{n})_{n\in\mathbb{N}}$ of elements in the field $\mathbb{F}_{2}$ with two elements. The injective hull of $R$ as a right $R$-module is ${\textstyle\prod\limits_{n\in\mathbb{N}}} \mathbb{F}_{2}$ and it has only one automorphism, namely, the identity. So $R_{R}$ is automorphism invariant. But $R$ is not right self-injective and hence not quasi-injective.

Another generalization of quasi-injectivity leads to the concept of continuous modules. A right $R$-module $M$ is said to be \textit{continuous }if every submodule of $M$ is essential in a direct summand of $M$, and any submodule of $M$ isomorphic to a direct summand of $M$ is itself a direct summand of $M$. It is well-known that every quasi-injective module is continuous, but not conversely (see \cite{MM}).

The following proposition shows that an automorphism invariant module need not be continuous and conversely a continuous module need not be automorphism invariant. We are thankful to Ashish Srivastava for the proof which in some way  implicit in his joint paper \cite{ESS}.

\begin{proposition}
A right module $M$ over a ring $R$ is both continuous and automorphism invariant if and only if $M$ is quasi-injective.
\end{proposition}

\begin{proof}
Suppose $M$ is continuous and  also automorphism invariant. We wish to show that $M$ is a fully invariant submodule of its injective hull $E(M)$. Let $f\in End(E(M))$. It is well-known (see \cite{lam}) that the endomorphism rings of injective modules are clean rings. So we can write $f=e+u$, where $e$ is an idempotent and $u$ is an automorphism. As $M$ is continuous, $M$ is invariant under $e$ and as $M$ is automorphism-invariant, it is invariant under $u$. This makes $M$ invariant under every endomorphism $f$ of $E(M)$. Hence $M$ is quasi-injective.
\end{proof}

It is interesting to observe that a Leavitt path algebra $L_{K}(E)$ is continuous if and only if it is automorphism invariant and in this case we can characterize $L_{K}(E)$ by means of graphical conditions on $E$ as indicated in the following corollary whose proof is an application of Theorem \ref{Regular Endo}.

\begin{corollary}
Let $E$ be a row-finite graph. Then the following properties are equivalent:
\begin{enumerate}[{\rm (i)}]
\item $L=L_{K}(E)$ is automorphism invariant as a right $L$-module;
\item $L$ is right continuous;
\item $L$ is left and right self-injective and von Neumann regular;
\item $E$ is acyclic and every infinite path ends in a sink.
\end{enumerate}
\end{corollary}

\begin{proof}
Since (iii) $\Longrightarrow$ (i) and (ii), (iii) $\Longleftrightarrow$ (iv), by Theorem \ref{Regular Endo}, we need only  prove that (i)  $\Longrightarrow$ (iii) and (ii) $\Longrightarrow$ (iii). Suppose $L$ is automorphism invariant as a right $L$-module. Let $A$ be the endomorphism ring of $L$. By \cite[Proposition 1]{GA}, $A/J(A)$ is von Neumann regular where $J(A)\ $is the Jacobson radical of $A$. By Lemma \ref{J(A)=0}, $J(A)=0$ and so $A$ becomes a von Neumann regular ring. Then, by Theorem \ref{Regular Endo}, $L$ is both left and right self-injective and von Neumann regular. This proves (iii). The preceding argument also implies (ii) $\Longrightarrow$ (iii) due to the fact that if $L$ is right continuous with endomorphism ring $A$, then again $A/J(A)$ is von Neumann regular (see \cite[Proposition 3.5]{MM}).
\end{proof}

\section{Strongly $\pi$-regular and Self-Injective Endomorphism Rings}

Recall that a (not necessarily unital) ring $R$ is called \emph{left (resp. right) $m$-regular} if for each $a\in R$ there exists $b\in R$ such that $a^m=ba^{m+1}$ (resp. $a^m=a^{m+1}b$). We say that $R$ is \emph{strongly $m$-regular} if it is both left and right $m$-regular. Also, in this context, we say that the ring $R$ is \emph{left (resp. right) $\pi$-regular} if for each $a\in R$ there exist $n\inÊ{\mathbb N}$ and $b\in R$ such that $a^n=ba^{n+1}$ (resp. $a^n=a^{n+1}b$), and $R$ is said \emph{strongly $\pi$-regular} if it is both left and right $\pi$-regular.

\par

Let $E$ be a row-finite graph, $L=L_{K}(E)$ and $A$ be the endomorphism ring of $L$ as a right $L$-module. We show that $A$ is strongly $\pi$-regular if and only if $A\cong {\textstyle\prod\limits_{i\in\Lambda}} \M_{n_{i}}(K)$ where $\Lambda$ is an arbitrary index set, the $n_{i}$ are positive integers less than a fixed integer $m$. This is equivalent to the graph $E$ being acyclic, column-finite, having no infinite paths and such that there is a fixed positive integer $m$ satisfying that the number of distinct paths ending at any given sink in $E$ is less than $m$. In this case, $A$ becomes strongly $m$-regular. We next investigate when $A$ is left self-injective. Interestingly, this happens if and only if $A\cong{\textstyle\prod\limits_{i\in\Lambda}} \M_{n_{i}}(K)$ where $\Lambda$ is an arbitrary index set and the $n_{i}$ are positive integers. Thus strong $\pi$-regularity of $A$ implies that $A$ is left/right self-injective but not conversely.

We shall use the known result \cite[Lemma 2]{AR} that if $A$ is strongly $\pi$-regular, so is any corner $\varepsilon A\varepsilon$, where $\varepsilon$ is an idempotent.

\medskip

\begin{definition}\label{InfinitePath}
{\rm
Let $E$ be a graph. A \emph{right infinite path} (also called infinite path in the literature) is $e_1\dots e_n\dots$, where $e_i\in E^1$ and $r(e_i)=s(e_{i+1})$ for every $i$. A \emph{left infinite path} is $\dots e_n \dots e_1$, where $e_i\in E^1$ and $s(e_{i+1})=r(e_{i})$.}
\end{definition}

\begin{lemma}
\label{Necessary Condition}Suppose $A$ is strongly $\pi$-regular. Then the graph $E$ is acyclic and does not contain neither right nor left infinite paths.
\end{lemma}

\begin{proof}
First we show that $L$ itself is strongly $\pi$-regular. It is enough if we show that it is right $\pi$-regular. Let $a\in L$. Then there exist $f\in A$ and an integer $n\geq1$ such that $(\lambda_{a})^{n}=\lambda_{a^{n}}=\lambda_{a^{n+1}}f$. Let $v$ be an idempotent in $L$ such that $va=a=av$. Then $\lambda_{a^{n}}=f\lambda_{va^{n+1}}=f\lambda_{v}\lambda_{a^{n+1}}=\lambda_{f(v)}\lambda_{a^{n+1}}$, by Lemma \ref{Lemma 1)}. Thus $L$ is strongly $\pi$-regular and by (\cite[Theorem 1]{AR}), the graph $E$ is then acyclic.

We claim that there are no infinite paths in $E$. Suppose $e_{1}e_{2}\cdot\cdot\cdot$ is a right infinite path in $E$ with $s(e_{i})=v_{i}$ for all $i$. Write $L=({\textstyle\bigoplus\limits_{i=1}^{\infty}} v_{i}L)\oplus Y$, where $$Y={\textstyle\bigoplus\limits_{u\in E^{0}\backslash\{v_{i}\ | \ i=1,2,...\}}} uL.$$ Define an endomorphism $f$ of $L$ by setting $f(Y)=0$, $f(v_{1}L)=0$ and, for each $i$, we have $f:v_{i+1}L\longrightarrow v_{i}L$ given by $f(v_{i+1}x)=e_{i}v_{i+1}x$. Then for all $n\geq1$, $f^{n}\neq0$ and there is no $h\in A$ such that $f^{n}=hf^{n+1}$ since $Ker(f^{n})=Y\oplus v_{1}L\oplus \cdot\cdot \cdot\oplus v_{n}L$ and $Ker(f^{n+1})=Y\oplus v_{1}L \oplus\cdot\cdot\cdot\oplus v_{n}L\oplus v_{n+1}L$. This contradiction shows that $E$ has no right infinite paths.

Suppose $E$ has a left infinite path $\cdot\cdot\cdot e_{n}\cdot\cdot\cdot e_{2}e_{1}$ with $u_{i}=r(e_{i})$ for all $i$. Write $L=({\textstyle\bigoplus\limits_{i=1}^{\infty}} u_{i}L)\oplus X$, where $$X={\textstyle\bigoplus\limits_{u\in E^{0}\backslash\{u_{i}:i=1,2,...\}}} uL.$$ As before, define an endomoprism $g$ of $L$ by setting $g(X)=0,g(u_{1}L)=0$ and, for each $i$, $g:$ $u_{i+1}L\longrightarrow u_{i}L$ by $g(u_{i+1}a)=e_{i}^{\ast}u_{i+1}a$. Then, for all $n\geq1$, $g^{n}\neq0$ and $g^{n}\neq hg^{n+1}$ for any $h\in A$. This contradiction shows that there are no left infinite paths in $E$.
\end{proof}

We are now ready to prove the first main theorem of this section.

\begin{theorem}\label{st Pi Regular}
Let $E$ be a row-finite graph and let $A$ be the endomorphism ring of $L$ as a right $L$-module. Then the following properties are equivalent:
\begin{enumerate}[{\rm (i)}]
\item $A$ is strongly $m$-regular for some positive integer $m$.
\item $A$ is strongly $\pi$-regular.
\item $L\cong\oplus_{i\in\Lambda}\M_{n_{i}}(K)$, where $\Lambda$ is an arbitrary index set and the $n_{i}$ are positive integers satisfying $n_{i}\leq m$ for a
fixed integer $m$ and for all $i$.
\item $E$ is acyclic, has no right nor left infinite paths and there is a fixed positive integer $m$ such that the number of paths ending in any given sink in
$E$ is $\leq m$.
\end{enumerate}
\end{theorem}

\begin{proof}
Now  (i) $\Longrightarrow$ (ii) is obvious.

\par

(ii) $\Longrightarrow$ (iii). Suppose $A$ is strongly $\pi$-regular. By Lemma \ref{Necessary Condition}, the graph $E$ is acyclic and has no infinite paths. Then, by Theorem \ref{Regular Endo}, $L$ is semisimple, say $L=\bigoplus_i L_{i}$ where each $L_{i}$ is a homogeneous component being a direct sum of isomorphic simple modules and is a two-sided ideal of $L$. If $L_{i}$ is not finitely genererated, then we can procced, as was done in the proof of Lemma \ref{Necessary Condition}, to construct an endomorphism $f$ of $L$ such that for all $n\geq1$, $f^{n}\neq0$ and $f^{n}\neq hf^{n+1}$ for any $h\in A$. Thus each $L_{i}$ is finitely generated. 

By Theorem 4.2.11, \cite{AAS} (see also \cite{AASM}), $L=\bigoplus_{i\in\Lambda}\M_{n_{i}}(K)$, where the $n_{i}$ are positive integers and $\oplus$ is a ring direct sum. We claim that there is a fixed positive integer $m$ such that $n_{i}<m$ for all $i$. Suppose not. Using the fact that each $\M_{n_{i}}(K)$ is a left/right artinian ring, choose, for each $i$, an $f_{i}\in \M_{n_{i}}(K)$ and a smallest integer $k_{i}\geq n_{i}/2$ (depending on $f_{i}$) such that $f_{i}^{k_{i}}=f_{i}^{k_{i}+1}a_{i}$ for some $a_{i}\in \M_{n_{i}}(K)$. Let $f=(\cdot\cdot\cdot,f_{i},f_{i+1},\cdot\cdot\cdot)$. Then coordinate multiplication on the left makes $f$ an endomorphism of $L$, but $f^{n}\neq f^{n+1}a$ for any $a\in A$ and any positive integer $n$. This contradiction shows that there is an upper bound $m$ for the integers $n_{i}$.

\par

(iii) $\Longrightarrow$ (i). Suppose $L\cong\bigoplus_{i\in\Lambda}\M_{n_{i}}(K)$, where $\oplus$ is a ring direct sum and the $n_{i}$ are integers less than a fixed positive integer $m$. Clearly, $A\cong{\textstyle\prod\limits_{i\in\Lambda}} \M_{n_{i}}(K)$. Now each $\M_{n_{i}}(K)$, being an artinian ring and having $n_{i}<m$, is clearly strongly $m$-regular. Then the direct product $A\cong{\textstyle\prod\limits_{i\in\Lambda}} \M_{n_{i}}(K)$ is also strongly $m$-regular.

\par

(iii) $\Longrightarrow$ (iv). Now $L\cong\bigoplus_{i\in\Lambda}\M_{n_{i}}(K)$, where the $n_{i}$ are positive integers less than a fixed integer $m$ and each $\M_{n_{i}}(K)$ is a two-sided ideal of $L$. Recognizing that $L=Soc(L)$ and thus is generated by its line points (see \cite{AMMS}), we conclude that every $\M_{n_{i}}(K)$ is the ideal generated by a sink $v_{i}$ in $E$ and that $n_{i}$ is the number of paths ending at this sink. It is also clear that the graph $E$ is acyclic and there are no infinite paths. By hypothesis, there is a positive integer $m$ such that $n_{i}<m$ for all $i$. This proves (iv).

\par

(iv) $\Longrightarrow$ (iii). By \cite{AAPS} we have $L_{K}(E)\cong\bigoplus_{i\in\Lambda}\M_{n_{i}}(K)$ where $n_{i}\in\mathbb{N}$ and every $\M_{n_{i}}(K)$ is the ideal generated by a sink $v_{i}$ and $n_{i}$ is the number of paths ending at this sink. Since, by hypothesis, $n_{i}<m$ for all $i$, the statement (iii) is immediate.
\end{proof}

\medskip

Strong $\pi$-regularity of $A$ and $L$ are closely connected as follows. 

\begin{proposition}\label{strongPiReg}
Let $E$ be a graph and $K$ any field. Then $A$ is strongly $\pi$-regular implies $L_K(E)$ is strongly $\pi$-regular. The converse is false.
\end{proposition}
\begin{proof} We will give a direct proof here although the first statement also follows from \cite[Theorem 1]{AR}. Suppose $A$ strongly  $\pi$-regular.  Take $x\in L_K(E)$ and let $u$ be a local unit such that $xu=ux$. We know that there exists $f\in A$ and an integer such that $x^n=fx^{n+1}$. Then $x^n=fux^{n+1}$. Since $L_K(E)$ is a left ideal of $A$ the element $fu$ belongs to $L_K(E)$ and we have proved that $L_K(E)$ is left $\pi$-regular. Now, use Lemma \ref{Necessary Condition} to show that $E$ does not contain infinite paths. 

The converse is false by using Theorem \ref{st Pi Regular} and \cite[Theorem 1]{AR}.
\end{proof}

\medskip

Our next theorem describes conditions under which the endomorphism ring $A$ is left self-injective. In its proof, we need the following lemma which is an easy generalization of a well-known result on vector spaces. We give the proof for the sake of completeness.

\begin{lemma}
\label{No Left injective}Let $M$ be a direct sum of infinitely many isomorphic simple right modules over a ring $R$. Then the  ring $S=End(M_{R})$ of endomorphisms of the right $R$-module $M$ is not left self-injective.
\end{lemma}

\begin{proof}
We first show that $M$ is isomorphic to a direct summand of $S$ as a left $S$-module. The proof is similar to the case of vector spaces (see \cite[Example 3.74B]{lam}). Write $M=\bigoplus_{i\in I} x_{i}R$ where $x_{i}R$ are simple right $R$-modules and let, for each $i$, $\pi_{i}:M\longrightarrow x_{i}R$ be the coordinate projection. We will use the easily established fact that given $i$ and given any $x\in M$, there is an endomorphism $f$ of $M$ such that $f(x_{i})=x$. Fix an index $i\in I$. Define $\alpha:S\longrightarrow M$ by $\alpha(g)=g(x_{i})$ for all $g\in S$. It is then easy to verify that $\alpha$ is a left $S$-module morphism which is actually an epimorphism by the fact stated  above. Define the morphism $\beta:M\longrightarrow S$ by setting $\beta(g(x_{i}))=g\pi_{i}$ for all $g\in S$. Then $\alpha\beta=1_{M}$, the identity on $M$, and hence $M$ is isomorphic to a direct sumand of $S$. Since $M$ is a direct sum of infinitely many non-zero submodules, we appeal to   \cite[Theorem 1]{S} to conclude that $M$ is not injective as a left $S$-module. This implies that $S$ is not left self-injective.
\end{proof}

\begin{theorem}\label{Self-injective}
Let $E$ be an arbitrary graph and $A$ be the endomorphism ring of $L=L_{K}(E)$ as a right $L_{K}(E)$-module. Then the following conditions are equivalent:
\begin{enumerate}[{\rm (i)}]
\item $A$ is left self-injective.
\item $L$ is semisimple and every homogeneous component is an artinian ring, concretely, $L \cong{\textstyle\bigoplus\limits_{i\in\Lambda}} \M_{n_{i}}(K)$, where every $n_{i}$ is an integer (the set of $n_i$'s might not be bounded).
\item $E$ is both row-finite and column-finite, is acyclic and there are no left or right infinite paths in $E$.
\end{enumerate}
\end{theorem}

\begin{proof}
(i) $\Longrightarrow$ (ii). We first show that the graph $E$ is row-finite. Since, by Proposition \ref{J(A)=0}, the Jacobson radical $J(A)=0$, we appeal to  \cite[Corollary 13.2]{lam} to conclude that $A$ is von Neumann regular. By Proposition \ref{regularityCond} (ii), $L$ is also von Neumann regular and so $E$ is acyclic by \cite{AR}. Then \cite[Proposition 4.4]{ARS} can be used to get that $E$ is row-finite. The main argument needed in the proof of \cite[Proposition 4.4]{ARS} is that $Lv$ is injective as a left $L$-module. This is true since every left ideal $L$ is a left ideal of $A$ (Lemma \ref{Left Ideal}) and for each vertex $v$ in $L$, $Lv=Av$, being a direct summand of $A$, is injective as a left $A$-module.

Since $E$ is row-finite and $A$ is von Neumann regular, we appeal to Theorem \ref{Regular Endo} to conclude that $L$ is semisimple as a right $L$-module.\ Write $L=\bigoplus_{i}L_{i}$, where the $L_{i}$ are the homogenous components of $L$ and the $L_{i}$ are ideals. Now $A\cong{\textstyle\prod\limits_{i}}End(L_{i})$ and so, for each $i$, $End(L_{i})$ is left self-injective. By Lemma \ref{No Left injective}, we conclude that each simple ring $L_{i}$ is a direct sum of finitely many isomorphic simple right $L$-modules and being artinian, $L_{i}=\M_{n_{i}}(K)$ for some integer $n_{i}$. This proves (ii).

\par

(ii) $\Longrightarrow$ (i). Suppose $L={\textstyle\bigoplus\limits_{i\in\Lambda}} \M_{n_{i}}(K)$, where every $n_{i}$ is an integer and ${\textstyle\bigoplus}$ is a ring direct sum. Then $A\cong {\textstyle\prod\limits_{i\in\Lambda}} \M_{n_{i}}(K)$ is a ring direct product of left self-injective rings $\M_{n_{i}}(K)$ and hence itself is left-self-injective (see \cite[Corollary (3.11B)]{lam}).

\par

(ii) $\Longrightarrow$ (iii). If $L={\textstyle\bigoplus\limits_{i\in\Lambda}} \M_{n_{i}}(K)$, where every $n_{i}$ is an integer and ${\textstyle\bigoplus}$ is a ring direct sum,  repeat the argument used in (iii) $\Longrightarrow$ (iv) of Theorem \ref{st Pi Regular} to conclude that $E$ is acyclic, each of the integers $n_{i}$ is the number of paths ending at a sink $v_{i}$ in $E$ and that $E$ has neither right nor left infinite paths so that every vertex connects to a sink. Since the $n_{i}$ are integers, it is also clear that $E$ is column-finite. Also from the proof of (i) $\Longrightarrow$ (ii), the graph $E$ is row-finite.

\par

(iii) $\Longrightarrow$ (ii). An application of \cite[Theorem 4.2.11]{AAS} gives that the Leavitt path algebra is semisimple. Since every homogeneous component is generated by a sink and since there are no infinite paths, there are no infinite sinks; this implies that every homogeneous component is isomorphic to $\M_{n}(K)$ for some positive integer $n$, which is clearly artinian.
\end{proof}

\begin{remark} {\rm From Theorems \ref{st Pi Regular} (iii) and \ref{Self-injective} (ii) it is clear that if the endomorphism ring $A$ is strongly $\pi$-regular, then $A$ is necessarily left self-injective, \ but not conversely. For an easy justification by means of graphs, consider the infinite graph given by 
$$
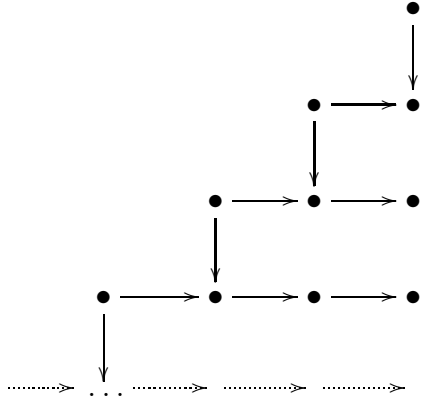
\xymatrix{  &  &  &   &   {\bullet} \ar[d] \\ &  &  & {\bullet} \ar[r] \ar[d] & {\bullet} \\  &  &  {\bullet} \ar[r] \ar[d] & {\bullet} \ar[r] & {\bullet} \\
   & {\bullet} \ar[r]  \ar[d] & {\bullet} \ar[r] & {\bullet} \ar[r] & {\bullet} \\ \ar@{.>}[r] &Ê\ar@{.>}[r] \dots &Ê\ar@{.>}[r] &Ê\ar@{.>}[r] &  }$$
This graph is row-finite, column-finite, acyclic, has neither left nor right infinite sinks and there exist infinitely many sinks $v_i$ where the $n_i$'s, the number of paths ending at $v_i$, are unbounded. Thus, $L_{K}(E)\cong{\bigoplus_{j=1}^{\infty}} \M_{n_{j}}(K)$ where $\{n_{j} \ \vert\ j\geq1\}$ is an unbounded set of integers, is left/right self injective, but it is not strongly $\pi$-regular as the integers $n_{j}$ are not bounded.}
\end{remark}

\section*{Acknowledgements}
Part of the work in this paper was done when the first author visited the Department of Mathematics at the University of Colorado at Colorado Springs during Summer 2013 and when the second author visited the Departamento de \'Algebra, Geometr\'\i a y Topolog\'\i a, Universidad de M\'alaga during May 2013. They gratefully acknowledge the kind hospitality of the faculty of these departments and the support of these institutions.

The first and third authors have been partially supported by the Spanish MEC and Fondos FEDER through project MTM2010-15223, by the Junta de Andaluc\'{\i}a and Fondos FEDER, jointly, through projects FQM-336 and FQM-7156.

\end{document}